\documentclass[12pt]{amsart}

\usepackage{amsfonts, amssymb, amsmath,
amsthm, latexsym, euscript, epic, eepic}
 \usepackage{graphicx}
\usepackage[nobysame]{amsrefs}
 \usepackage{hyperref}
\usepackage{a4wide}

\usepackage{xcolor}

\newtheorem{maintheorem}{Theorem}

\newcommand{\cmt}{\begin{maintheorem}}
\newcommand{\fmt}{\end{maintheorem}}

\newtheorem{theorem}{Theorem}[section]
\newtheorem*{theorem*}{Theorem}
\newtheorem{prop}[theorem]{Proposition}
\newtheorem{lemma}[theorem]{Lemma}

\theoremstyle{definition}
\newtheorem{defn}[theorem]{Definition}

\newtheorem{remark}[theorem]{Remark}

\def\Q{\mathcal Q}

\def\DD{\mathfrak D}

\def\o{\omega}

\def\e{\epsilon}

\subjclass[2010]{37A10, 37C75, 37D25, 37E05}


\author{Jos\'e F. Alves}
\address{Jos\'e F. Alves, Departamento de Matem\'atica, Faculdade de Ci\^encias da Universidade do Porto, Rua do Campo Alegre s/n, 4050-637 Porto, Portugal}
\email{jfalves@fc.up.pt}
\urladdr{https://www.fc.up.pt/pessoas/jfalves}

\author{Dalmi Gama}
\address{Dalmi Gama,  Campus Universit\'ario do Tocantins/Camet\'a,
Pr\'edio da Administra\c c\~ao Professora Maria Cordeiro de Castro (2º andar),
Trav. Pe. Ant\'onio Franco 2617 - Bairro da Matinha,
CEP 68400-000    Camet\'a-PA, Brazil}
\email{dalmi@ufpa.br}
\urladdr{http://docente.ufpa.br/dalmirmat}

\author{Stefano Luzzatto}
\address{Stefano Luzzatto, Abdus Salam International Centre for Theoretical Physics (ICTP), Strada Costiera 11, 34151 Trieste, Italy}
\email{stefano@ictp.it}
\urladdr{https://www.stefanoluzzatto.net}

\thanks{JFA was partially supported by  CMUP (UID/MAT/00144/2019), PTDC/MAT-PUR/28177/2017 and   PTDC/MAT-PUR/4048/2021, which are funded by FCT (Portugal) with national (MEC) and European structural funds through the program  FEDER, under the partnership agreement PT2020. The authors are grateful for the support and hospitality of the Abdus Salam International Centre for Theoretical Physics and the Department of Mathematics of the Faculty of Sciences of the University of Porto, where this work was carried out. }

\title[Statistical stability of interval maps]
{Statistical stability of interval maps 
\\ with critical points and singularities}

\date{\today}

\begin{document}

\begin{abstract}
  We prove strong statistical stability of a large class of  one-dimensional  maps which may have  an arbitrary
finite  number of discontinuities and of non-degenerate critical points and/or singular points with infinite derivative, and satisfy some expansivity and bounded recurrence conditions. This generalizes known results for maps with critical points and bounded derivatives and in particular proves statistical stability of Lorenz-like maps with critical points and singularities studied in \cite{LuzTuc99}. We introduce a  natural metric on the space of maps with discontinuities which does not seem to have been used in the literature before.

%\setcounter{tocdepth}{1}

 % \tableofcontents
 
\end{abstract}

\maketitle

\section{Introduction and statement of results}\label{sec_intro}

\subsection{Motivation}
It has been understood for a long time that even  deterministic dynamical systems are not necessarily predictable since, as Poincar\'e \cite{Poi03}  wrote as far back as 1903: ``\emph{it may happen that small differences in initial conditions produce very great one in the final phenomena. [...] Prediction becomes impossible and we have the fortuitous phenomena}''. The  significance of this observation was arguably not fully appreciated until the advent of computers and the work of the meteorologist Lorenz~\cite{Lor63} in the 1960s, who showed how \emph{quickly} this unpredictability can arise. Lorenz coined the term \emph{sensitive dependence on initial conditions} amid the general realization by many physicists that a huge number of naturally arising systems exhibit this characteristic and are therefore, in some sense, \emph{chaotic} and \emph{unpredictable}.

Around the same time, several mathematicians and mathematical physicists, most notably Ruelle and Sinai,  were building on ideas from statistical physics and developing mathematical tools and techniques which they used to prove the extremely remarkable result that, notwithstanding the sensitive dependence on initial conditions,  some very chaotic systems actually exhibit a strong form of \emph{statistical predictability}, in the sense that there exists a probability measure \( \mu \) which describes the asymptotic distribution in space of almost every initial condition, and therefore the \emph{statistics} of the system are \emph{independent of initial conditions}. More formally, given a map \( f: X \to X \) and an initial condition \( x \in X \) we can describe the first \( n \) points in the orbit  of \( x \) with the probability measure 
\[
\mu_{n}(x) :=  \frac 1n \sum_{i=0}^{n-1}\delta_{f^{i}(x)}.
\]
If this sequence of probability measures  converges, i.e. if there exists a limiting probability measure \( \mu \) such that 
\(
\mu_{n}(x) \to \mu
\)
in the weak--star topology as \( n \to \infty\), then for all sufficiently large \( n \) the probability measure \( \mu \) describes up to arbitrarily small errors the statistical distribution of the orbit  of \( x \). We can then define the \emph{basin} 
\(
\mathcal B_{\mu}:= \{x\in X: \mu_{n}(x) \to \mu\}
\)
of \( \mu \) as the set of all points whose statistical distribution is given by \( \mu\) and we say that \( \mu \) is a \emph{physical measure} if \( \mathcal B_{\mu}\) has full, or at least positive, Lebesgue measure. Physical measures do not always exist, the identity map is an easy counterexample, but over the last few decades there has been a large amount of research aimed at establishing the existence \cites{Alv00,Alv20,AlvBonVia00,AlvDiaLuz13,AlvDiaLuz17,AlvLep15,AlvLi15,AlvLuzPin05,AlvPin10,AraPac09,BenYou93,BerZwe13,BluYou19,Bow75,BowRue75,BruDemMel10,BruKelNow96,Bur21,Buz00,BuzCroSar22,Car93,Che99,CliDolPes16,CliLuzPes23,CliLuzPes17,CroYanZha20,Fro98,GanYanYanZhe21,Her18,HofKel82,JakNew95,Led84,Lep98,MarWin16,Met00,PesSin82,Pin06,RivShe10,Rue76,Sin72,Tak12,Tal20,Tsu05,Vec22,Via99,ViaYan13,You02} (or non-existence \cites{AraPin21,BarKirNak20,BerBie22,ColVar01,CroYanZha20,Her18,HofKel95,HofKel90,Kel04,KirLiNak22,KirNakSom22,KirNakSom21,KirNakSom19,KirSom17,KirSom17,Tak94,Tak08,Tal20,Tal20}) of physical measures and the well-known \emph{Palis conjecture} asserts that ``typical'' systems have a finite number of physical measures and that almost every point is in the basin of some physical measure. 

One  reason for which the existence of physical measures is so remarkable and unexpected is that chaotic dynamical systems generally have a very complicated topological structure, with infinitely many periodic points of arbitrarily large periods, points with dense orbits, and points with many different possible omega-limit sets. In some quite special situations such systems are \emph{structurally stable} in the sense that small perturbations of the system essentially preserve this structure and, in a suitable topology,  sufficiently nearby systems are topologically conjugate to each other. In most situations this is however not the case and arbitrarily small perturbations can give rise to multiple bifurcations and completely different topological structures. In both situations, an interesting and relevant question is whether physical measures are persistent and  whether they depend \emph{continuously} on the perturbation. This means that not only the statistics of the system are independent of the initial conditions but also that it is ``robust'' in the sense that small perturbations of the system will give rise to a system which may have a completely different topological structure but for which almost all orbits have a statistical distribution which is almost the same as that of the original system. We then say that the system is \emph{statistically stable}\footnote{We remark that the term  \emph{stochastic stability} can also be found in the literature and usually refers to quite general perturbations where one is allowed to add some ``noise'' or perturb the map in different ways \emph{at each iteration}, thus giving rise to a \emph{random system} to which the notion of \emph{stationary measure} applies. Stochastic stability then refers to the property that the stationary measures converge to the original measure of the unperturbed system as the size of the perturbation goes to 0. Statistical stability, on the other hand, considers the physical measures of nearby systems in some natural topology and refers to the property that  such measures are close. Statistical stability can be thought of as a special case of stochastic stability by choosing a ``constant noise'', i.e. applying the very same perturbation of the original map at each iteration, so that one is effectively iterating a nearby map. 
 }. 

The first result on statistical stability goes back to  \cite{Kel82}  and over the last 20 years there has been a considerable amount of work studying statistical stability in a variety of settings in both one-dimensional and higher-dimensional systems \cites{AlvKha19,Alv20,Alv04,AlvCarFre10b,AlvCarFre10,AlvPumVig17,AlvSou12,AlvSou14,AlvVia02,Ara21,BahRuz17,ColTre88,DemLiv08,Fre05,FreTod08,Gal17a,Gal18,GalLuc15,GalSor22,Kel82,Tal20,Tal22,Thu01,Tsu96,Tsu96,Vas07}, in some cases it is even possible to prove that the physical measure depends \emph{Lipschitz} or \emph{differentiably} on the perturbation, sometimes refereed to as \emph{linear response} \cites{AntFroGal22,BahGalNis18,BahRuzSau20,BahSau15, BalBenSch15, BalTod16,Bal14,Gal21,GalGiu17,GalPol16,GalSed20,GalSor22,Kor16,Rue09}.

In this paper we consider a large class of one-dimensional maps which were introduced in \cite{DiaHolLuz06} and to which none of the existing results in the literature apply. These maps may contain an arbitrary (finite) number of discontinuities, critical points, and/or singularities with unbounded derivatives, and include as special cases many well-studied families of maps such as quadratic maps, more general smooth unimodal and multimodal maps, Rovella maps, for which statistical stability has already been established in the references cited above. However, the presence of singularities with unbounded derivative and especially  the \emph{co-existence} of singularities and of critical points, gives rise to significant technical challenges. Some techniques for handling these cases were introduced  in \cite{DiaHolLuz06} where it was proved that under  some conditions to be defined below, there always exist some  physical measures. Under some slightly stronger conditions, which imply that this physical measure is unique, we will prove  that  such maps are statistically stable.

\subsection{Maps with critical points  and singularities}
%subsection{Non-degenerate critical and singular points.}\label{critsingdefn}
We now give the precise definition of the class of maps we consider. We consider a family of interval maps  as follows. 
Let~\( I \) be a compact interval and $f: I \to I$ a $C^2$ local
diffeomorphism outside a finite set $\mathcal{C}
\subset\mathrm{int}(J)$, of \emph{non-degenerate critical and
singular} points. These are points at which $f$ may be discontinuous
or the derivative of $f$ may vanish or be infinite. In order to
treat all possibilities in a formally unified way we consider
\newcommand{\quand}{\quad\text{ and } \quad }
\begin{equation}\label{eq:cpm}
f(c^{+}):=\lim_{x\to c^{-}}f(x)
\quand
f(c^{-}):=\lim_{x\to c^{+}}f(x)
\end{equation}
 as distinct
\emph{critical values}, thus implicitly thinking of $c^+$ and $c^-$
as distinct critical points. When referring to a neighbourhood of a
critical point, we shall always be referring to the appropriate
one-sided neighbourhood of that point. For simplicity, we also treat the ``regular'' critical points at which \( f\) is smooth  as two distinct critical points. 
%\begin{figure}[h]\label{map}
%    \includegraphics[width=5cm]{CritSing}
%    \caption{A map with a finite number of critical and singular points}
%\end{figure}
We say that the set of critical points \( \mathcal C \)  is  \emph{non-degenerate}  if  
%for each \( c \in \mathcal C \) there exists a constant $\ell_c\in (0,\infty) $ such that for
%each $x$ in a neighbourhood of $c$ we have
%\begin{equation}\label{defnCmap1}
%|f(x)-f(c)| \approx |x-c|^{\ell_c}, 
%\quad 
%|f'(x)| \approx |x-c|^{\ell_c-1} 
%\quad \text{ and } \quad 
% |f''(x)|\approx |x-c|^{\ell_c-2},
%\end{equation}
there is a constant $C>0$ such 
for each \( c \in \mathcal C \) there exists  $\ell_c\in (0,\infty) $ such that for
every $x$ in a neighbourhood of $c$ we have
\begin{equation}\label{defnCmap1}
\frac1C\le\frac{|f(x)-f(c)|}{  |x-c|^{\ell_c}}\le C, 
\quad 
\frac1C\le\frac{|f'(x)| }{ |x-c|^{\ell_c-1}}\le C 
\quad \text{ and } \quad 
\frac1C\le\frac{   |f''(x)|}{ |x-c|^{\ell_c-2}}\le C.
\end{equation}
%
%where the symbol \( \approx \) means that the ratio between the two quantities lies between \( C\) and \( C^{-1}\) for some \( C > 0 \).  
We let  $$\mathcal{C}^{c}=\{c:\ell_c\geq 1\}\quand\mathcal{C}^{s}=\{c:0<\ell_c<1\}$$ denote the set of critical and
singular points, respectively. Notice that two ``distinct'' points in
$\mathcal{C}^{c}$ and $\mathcal{C}^{s}$ may actually correspond
to the same point in $I$. 
%We let
%\(
%\ell = \max_{c\in\mathcal C_{c}}\{\ell_{c}\}\) and \(
%\ell^{*} = \max_{c\in\mathcal C_{s}}\{\ell_{c}\}.
%\)
When there is no possibility of confusion we will often use the term
``critical point'' to refer to a point of \( \mathcal C \) without
necessarily specifying if that point is really a critical point in the
traditional sense with \( \ell_{c}>1 \) or whether it is a singular
point with \( \ell_{c}\in (0,1) \), or a ``neutral'' point with \(
\ell_{c}=1 \). For
ease of exposition we assume the derivative of $f$ at the points of
discontinuity is either unbounded or zero. To accommodate bounded
derivatives, we would have to slightly modify our
argument to include the case of a return to a region where there is
a bounded discontinuity, the derivative growth and distortion
estimates would not be affected by such bounded discontinuities.
For any \( x\in I \) and
for small \( \delta>0 \), let
\begin{equation}\label{eq:Delta}
\DD(x)=\min_{c\in\mathcal{C}}|x-c|
%\quad \DD(\omega) = \sup_{x\in\omega}\{\DD(x)\} 
\quad \text{ and } \quad \Delta=\{x: \mathfrak D(x) \leq \delta\}.
\end{equation}
denote respectively the distance of \( x \) from the nearest critical point, and a $\delta$-neighbourhood
of $\mathcal{C}$. 
%\subsection{Dynamical assumptions}\label{H1toH3}
To define the family  \( \mathcal F \)  we suppose first of all that every map in \( \mathcal F \) satisfies the non-degeneracy conditions given above.  For clarity we will sometimes use a subscript, such as writing \( \Delta_{f}\) or \( \mathcal C_{f}\),  to make precise that the objects under consideration depend on the specific map \( f \in \mathcal F \). 
We assume that  
\begin{equation}\label{card}
 N_c := \# \mathcal C_f \text{ is constant in \(  \mathcal F \) }
 \end{equation} 
 and that there exist  uniform constants \( \hat\ell, \ell>0\) such that  \( \hat \ell > \ell_c >\ell\)   and  the constant \( C \)    in~\eqref{defnCmap1} can be chosen the same for all \( f\in \mathcal F \) and all critical points \( c\in \mathcal C_{f}\). 
If \( c\in \mathcal C_{f}\), then it is implicit that its images  \( c_{k}\) are given by iterates of the map \( f \). We moreover suppose that there exist constants  
\begin{equation}\label{eq:alpha}
 \lambda, \Lambda, \kappa,  \alpha, \delta> 0, \quad \text{ with } \quad \alpha < \lambda/5\hat\ell,
 \end{equation}
  such that for every \( f\in \mathcal F \) the following conditions hold.

\begin{itemize}
\item[\bf(H1)] \textbf{Expansion outside \(\mathbf \Delta_{f} \):}
\emph{For all \( x\in I \) and \( n\geq 1 \) such that
 $x,\dots, f^{n-1}(x)\not\in\Delta_{f}$, we have
 $$
|(f^n)'(x)|\geq \kappa \delta e^{\lambda n}.
$$
If, in addition, $x_{0}\in f(\Delta_{f})$ and/or  $f^n(x)\in\Delta_{f}$}, then  we have
 $$
|(f^n)'(x)|\geq \kappa e^{\lambda n}.
$$
\item[\bf(H2)] \textbf{Recurrence and exponential growth:} \label{H2} \emph{For all $
k\geq 1$ and all ${c\in\mathcal{C}^{c}_{f}}$,  we have
$$\DD(c_k)\geq\delta e^{-\alpha k}\quad\textrm{and}\quad
|(f^k)'(c_1)|\geq e^{\Lambda k}.$$}
\item[\bf(H3)] \textbf{Density of preimages:}
\emph{There exists  $c^*$ in $\mathcal{C}_{f}$ whose preimages are dense
in the interval $ I$ and do not intersect the set  $\mathcal{C}_{f}$.}
\end{itemize}

It  was  proved in \cite[Theorem 1]{DiaHolLuz06} that if (H1)-(H3) hold for constants $ \alpha, \delta$  sufficiently small with respect to   $ \lambda, \kappa$,
 then there exists a full branch induced Gibbs-Markov map with integrable return times. We will therefore assume that this holds and this completes our definition of the class \(  \mathcal F \). In particular, \emph{each \( f \in \mathcal F \) admits an ergodic invariant probability measure~\( \mu_{f}\) which is absolutely continuous with respect to Lebesgue measure \( m \).} To simplify the notation we will sometimes   use \( |A|\) to denote the Lebesgue measure of a Borel  set~\( A\).

 We note that  condition (H3) is more restrictive than that used in \cite{DiaHolLuz06}, where the preimages of $c^*$ are only assumed to be dense in some subset of~\( I \). Our condition implies in particular uniqueness of the probability measure \( \mu_{f}\) without which it would be much harder to even formulate the property of statistical stability, without adding much to the result. We also note that conditions (H1)-(H3) are non-trivial conditions which generally require non-trivial arguments to verify. There are however many  one-parameter families of maps in which they hold for  large (positive measure) sets of parameters, including 
\begin{itemize}
\item%[(E1)] 
\emph{families of  smooth   maps with critical points} \cite{BenCar91,BruLuzStr03};
\item%[(E2)]    
\emph{Lorenz-like families of maps with critical points}~\cite{AlvSou12, Rov93};
\item%[(E3)]   
\emph{Lorenz-like families of maps with
 critical points and singularities~\cite{LuzTuc99}}.
\end{itemize}

 \subsection{Statement of results}
 We are now ready to state our results.  We assume that 
 
 \medskip
 \begin{quote}
 \center{\em
 \( \mathcal F \) is a  family of maps satisfying conditions (H1)-(H3)
 }
 \end{quote}
 
 \medskip
 \noindent for sufficiently small constants \( \alpha, \delta\) as mentioned above. 
 First of all we will show that our strengthening of condition (H3) above implies a stronger version of the results of \cite{DiaHolLuz06} in that it implies uniqueness of the physical measure. 
 
 \begin{maintheorem}\label{thm:unique}
 %Let \( \mathcal F \) be a family of maps satisfying conditions (H1)-(H3). 
Every \( f\in \mathcal F \) has a unique ergodic invariant probability measure \( \mu_{f}\) which is absolutely continuous with respect to Lebesgue measure. 
 \end{maintheorem}
 
 We give the proof of the uniqueness part of Theorem \ref{thm:unique}  in Section \ref{sec:unique}, after recalling the main features of the argument in 
 \cite{DiaHolLuz06}.  
   The main focus of our paper here  is the question of whether the measures \( \mu_{f}\) depend continuously on the maps \( f\in  \mathcal F \). To formulate this  precisely we define a metric on the space \(  \mathcal F \). Notice that we cannot use the standard \( C^r\) metric because the maps are not differentiable or even continuous and, importantly, \emph{we are not assuming that the discontinuity points are the same for all maps in the family}. Thus we want to allow maps to be nearby as long as their points of discontinuities are close. To define this metric we first of all let 
\begin{equation}\label{eq:critorder}
 c_{1}^{f}< c_{2}^{f}< \cdots < c_{s}^{f}
\quad \text{  and } \quad c_{1}^{g}< c_{2}^{g}< \cdots < c_{s}^{g}
\end{equation}
be the critical/singular points for \( f \) and \( g \) respectively, and for each \( i=1,..,s\), let \( \ell_{i}^{f}, \ell_{i}^{g}\) denote the order of \( c_{i}^{f}, c_{i}^{g}\) respectively, as in \eqref{defnCmap1}. For any (small) \( \eta>0\), let 
\( \mathcal N_{2\eta, i}^{f,g} 
\) denote the intersection of the (two-sided) neighbourhoods  of radius \( 2 \eta\) of the critical points \( c_{i}^{f}\) and \(  c_{i}^{g}\),  and let 
\[
I_{\eta}:= I \setminus \bigcup_{i=1}^{s}\mathcal N_{2\eta, i}^{f,g}.
\]
Notice that the first and second derivatives of both \( f\) and \( g \) are bounded away from~0 and~$\infty$ in \( I_{\eta}\). We can now define a natural metric on \( 
\mathcal F\). 
 \begin{defn}
 For \( f,g \in \mathcal F \) we set 
\[
d(f,g)= \inf_{\eta >0}\left
\{\sup_{i}|c_{i}^{f}-c_{i}^{g}|< \eta, \;  \sup_{i}|\ell_{i}^{f}-\ell_{i}^{g}|< \eta, 
\; d_{2}(f|_{I_{\eta}}, g|_{I_{\eta}}) < \eta
\right\},
\]
where   \(d_{2}\) denotes the standard \( C^{2}\) distance.
\end{defn}
It is straightforward to verify that this defines a metric on \( \mathcal F \). 
Then, following \cite{AlvVia02} we recall the following definition. 

\begin{defn}\label{def:statstab}
The family \( \mathcal F\) is \emph{strongly statistically stable} if, for every \( f \in \mathcal F\),  the map 
\[
\mathcal F \ni f \longmapsto \frac{d\mu_{f}}{dm} 
\] 
is continuous with respect to the metric \( d \) in \( \mathcal F \) and the \( L^{1}\) norm in the space of densities. 
\end{defn}
Our main result is the following.

\begin{maintheorem}\label{thm:statstab}  
%Let \( \mathcal F \) be a family of maps satisfying conditions (H1)-(H3). Then 
\( \mathcal F\) is  strongly statistically stable.
\end{maintheorem}

As a consequence of Theorem~\ref{thm:statstab}, we obtain the statistical stability of the family of Lorenz-like maps with
 critical points and singularities corresponding to the (positive Lebesgue measure set of) parameters obtained   in~\cite{LuzTuc99}.

\begin{remark}

 It should be noted that this work was motivated by the impossibility of applying some existing results in the literature, such as \cite{Alv04,Fre05}, on the statistical stability of families of transformations with non-uniform expansion, to the families of applications considered in \cite{LuzTuc99}. Applications of \cite{Alv04,Fre05} were particularly successful for maps with criticalities, such as quadratic maps or Rovella maps, but made important use of the fact that the derivative is bounded (in particular to obtain estimates from below for the binding periods, depending on the depth of the return). This strategy cannot be carried out for the maps in \cite{LuzTuc99} and for the class of maps defined above due to the presence of  unbounded derivatives. 
 We will still apply the general results of \cite{AlvVia02} but this we will require some very subtle and non-trivial estimates in order to verify the required conditions. 
\end{remark}

\begin{remark}
We mention also that another  one-parameter family  of Lorenz-like maps with critical points and discontinuities, very similar to that considered in \cite{LuzTuc99},  is studied in \cite{LuzVia00}, where it is proved that there is a positive Lebesgue measure set of parameters for which conditions (H1)-(H3) can be shown to hold,  \emph{except} for the first part of (H2), i.e. the slow recurrence condition \( \DD(c_k)\geq\delta e^{-\alpha k} \). It turns out that this is crucial for the techniques we use here and indeed for all available techniques, including for the construction of the induced maps in \cite{DiaHolLuz06}. For these maps it is therefore still an open question whether they even admit an absolutely continuous invariant probability measure. 
\end{remark}

In Section \ref{sec:strategy} we explain the overall strategy of our argument and reduce the proof of Theorem \ref{thm:statstab} to two fundamental Propositions \ref{prop:main2} and \ref{prop:main}. In Section \ref{escape_part} we recall the notation and the main properties of the construction of the induced Gibbs-Markov maps constructed in \cite{DiaHolLuz06}. Then in Section \ref{sec:unique} we prove Theorem \ref{thm:unique} on the uniqueness of the physical measure, in Section \ref{sec:small} we prove Proposition \ref{prop:main2} and in Sections \ref{sec:small} and \ref{sec_bind} we prove Proposition \ref{prop:main}, thus completing the proof of Theorem \ref{thm:statstab}.

\section{Strategy of the proof} \label{sec:strategy}
We recall first of all the key technical result of \cite[Theorem 1]{DiaHolLuz06} is that there exists a (one-sided) neighbourhood  $\Delta^{*}\subset I$
of the critical point \( c^{*} \),  a countable ($m$ mod~0)
partition  $\mathcal{P}$ of \( \Delta^{*} \) into
       subintervals, a function \( T: \Delta^{*} \to \mathbb N \)
       defined almost everywhere and constant on elements of
       the partition \( \mathcal P \), and
       constants \( C,  D, \gamma >0 \) and \( \sigma> 1\),  such that
           \begin{equation}\label{eq:Cgammatail}
    |\{T>n\}|< C  e^{-\gamma n},
    \end{equation}
    and        for all \( \omega\in\mathcal P \) and \( T=T(\omega) \),
 the map \( f^{T}:\omega \to \Delta^{*} \) is a \( C^{2} \)
 diffeomorphism and for all \( x,y\in\omega \)
    \begin{equation}\label{ref:unif1}
    \left| \frac{(f^{T})'(x)}{(f^{T})'(y)} -1 \right|
    \le  \mathcal D |f^{T}(x)-f^{T}(y)|
\quad \text{ and } \quad 
|f^{T}(x)-f^{T}(y)|\ge \sigma |x-y|. 
\end{equation}
We will refer to the construction carried out in \cite{DiaHolLuz06}  to prove the following two key results.

\begin{prop}\label{prop:main2}
For every \( f\in \mathcal F \) and every  $N\geq 1$ and  $\e>0$, there is
$\eta=\eta(\e,N)>0$ such that for every \( g\in \mathcal F \) and  $j=1,\dots,N$
 \begin{equation*}
d(f,g)<\eta\implies 
|\{T_{f}=j\}\triangle \{T_{g}=j\}|<\e,
 \end{equation*}
where $\triangle$ represents the symmetric difference of two sets.
\end{prop}

\begin{prop}\label{prop:main}
For every \( f\in \mathcal F \), the constants $C, \mathcal D, \sigma, \gamma$ as
above may be chosen uniformly in a neighbourhood of  \( f\). 
\end{prop}

\begin{remark}
As we shall see, the constants \( \mathcal D\) and \( \sigma\) can actually be chosen uniformly for the whole family \( \mathcal F\). For the other constants we are only able to prove local uniformity, mainly  due to the argument required to define the constant \( \xi \) in Lemma~\ref{retpart} below. 
\end{remark}

\begin{proof}[Proof of Theorem \ref{thm:statstab}]
We show that Propositions \ref{prop:main2} and \ref{prop:main} allow us to apply \cite[Theorem A]{AlvVia02}, whose conclusion is exactly the strong statistical stability required in our Theorem \ref{thm:statstab}. This 
 relies on three  assumptions (U1)-(U3) and we explain below why the conclusions of Propositions \ref{prop:main2} and \ref{prop:main} imply these assumptions. 

Assumption (U1) is precisely Proposition \ref{prop:main2} except for the fact that the \( C^{k}\) distance in (U1) is replaced by our metric defined above. This change is completely inconsequential since the only fact used in the proof of \cite[Theorem A]{AlvVia02} is the fact that  \[|\{T_{f}=j\}\triangle \{T_{g}=j\}|<\e\] for ``nearby'' maps \( f,g \), whatever the metric. 

Assumption (U2) can be formulated in our setting as follows: for all   $\e>0$, there exists  $N\geq 1$ such that for all \( f \in \mathcal F\) 
\begin{equation}\label{eq:tail}
\sum_{j=N}^\infty  |\{T_f>j\}|<\e.
\end{equation}
The precise formulation in \cite{AlvVia02} is a little more general because the result there allows for maps on higher dimensional manifolds, in which case the uniform summability condition~\eqref{eq:tail} is formulated as 
 $
\|\sum_{j=N}^\infty\mathcal X_{\{T>j\}}\|_q<\e,
 $
where $\mathcal X_{\{T>j\}}$ denotes the characteristic function of the set ${\{T>j\}}$,  \( q \) is the conjugate exponent to \( p:=d/(d-1) \),  and \( d \) is the dimension of the manifold. In our case, this gives \( p=\infty\) and \( q=1\) and therefore the condition  $
\|\sum_{j=N}^\infty\mathcal X_{\{T>j\}}\|_q<\e,
 $ gives \eqref{eq:tail}, which clearly follows from \eqref{eq:Cgammatail}  and the uniformity of the constants \( C, \gamma\). 
 
 Finally, assumption (U3) is a statement about the uniformity of certain constants  $\sigma$,~$K$,~$\beta$ and $ \rho$. The constants \( \sigma \)  and \( K \) are  precisely our constants \( \Lambda^{-1}\) and  \( \mathcal D \), respectively, which Proposition \ref{prop:main} says can be chosen uniformly in \( \mathcal F \). Assumption  (U3), as stated in  \cite{AlvVia02},  involves two additional constants \( \beta, \rho\) which are however only required in the more general setting of maps on  higher dimensional manifolds which may not be full branch. 
\end{proof}
%
%It remains to prove Propositions \ref{prop:main2} and \ref{prop:main}. We will prove this in the following two sections. In Section~\ref{sec:small} we prove Proposition \ref{prop:main2}  and in Section \ref{sec:uniform} we prove Proposition \ref{prop:main}. 

\section{The induced map}\label{escape_part}

The key point in the proof of Propositions \ref{prop:main2} and \ref{prop:main}  is to show some continuity and uniformity properties of the construction of the induced map in \cite{DiaHolLuz06} which are not immediately clear. We therefore recall here the notation, the main steps and properties of the construction, for more details, remarks, and proofs we refer the reader to \cite{DiaHolLuz06}.

\subsection{Critical partitions and binding periods}
\label{critpart}
We recall first from \eqref{eq:cpm} that each critical point comes with a \emph{one-sided neigbourhood} and we use the notation \( c^{\pm}\) depending on whether this is a left or right neighbourhood. 
For each $c\in\mathcal{C}$ and for any integer $r\geq 1$ we let
\[
I_{r}(c)= [c+e^{-r},c+e^{-r+1}) \quad \text{and} \quad
I_{-r}(c)=(c-e^{-r+1},c-e^{-r}].
\]
We suppose without loss of
generality that 
\begin{equation} \label{rdelta}
r_{\delta} := \log \delta^{-1} \in\mathbb N.
\end{equation}
For
each $c\in\mathcal{C}$, let
\begin{equation*}
\Delta_{c}=
\begin{cases}
\{c\}\cup \bigcup_{r\geq r_{\delta}+1} I_{r}(c), &\text{if $c=c^{+}$},\\
\{c\}\cup \bigcup_{r\leq -r_{\delta}-1} I_{r}(c), &\textrm{if
$c=c^{-}$;}
\end{cases}
\quad
%\end{equation*}
%and
%\begin{equation*}
\hat \Delta_{c}=
\begin{cases}
\{c\}\cup \bigcup_{r\geq r_{\delta}} I_{r}(c), &\text{if $c=c^{+}$},\\
\{c\}\cup \bigcup_{r\leq -r_{\delta}} I_{r}(c), &\textrm{if
$c=c^{-}$}.
\end{cases}
\end{equation*}
Notice that \( \hat\Delta_{c} \) is just \( \Delta_{c} \) union an
extra interval of the form \( I_{\pm r_{\delta}} \) and that the set~\( \Delta \) defined in \eqref{eq:Delta} is just the union of all \( \hat\Delta_{c}\) with \( c\in \mathcal C \). 

We further subdivide each $I_r\subset \Delta$ (and not the additional \(
I_{\pm r_{\delta}}\subset \hat\Delta\setminus\Delta \))
into $r^2$ intervals $I_{r,j}$,
$j\in [1,r^2]$ of equal length. The intervals \(  I_{r,j}\) together with the extreme intervals 
\( I_{\pm r_{\delta}} \) define what we call the \emph{critical
partition}  $\mathcal{I}$ of \( \Delta \). Finally, for each
\( r\geq r_{\delta}+1, \) and \( j\in [1,r^2] \), let
\( \hat I_{r} \)
denote the union of \( I_{r} \) and its two neighbouring intervals. In
particular, if \( I_{r,j}=I_{r_{\delta}+1, (r_{\delta}+1)^{2}} \) is
one of the two extreme intervals of \( \Delta \), then \( \hat I_{r,j} \)
denotes the union of this interval with the adjacent intervals \(
I_{r, j-1} \) and \( I_{r_{\delta}} \) (which has not been subdivided
into subintervals).

Using the partitions defined above, we formalize the notion of a
\emph{binding period} during which points in the critical region \(
\Delta \) \emph{shadow} the orbit of the critical point.
For each \( r\geq r_{\delta}+1 \), \(
I_{r}\in \mathcal I \) belonging to the component of \( \Delta \)
containing a critical point \( c\in \mathcal C \), set
\begin{equation}\label{ss:binding}
p(r)=\begin{cases} 0, &\text{ if } c\in\mathcal{C}_{s},\\
\max\left\{k:  |f^{j+1}(x) - f^{j+1}(c)|\leq \delta e^{-2\alpha j} \
\forall\ \ x\in \hat I_{r}, \ \forall \  j  \leq k\right\},
&\text{ if } c\in\mathcal{C}_{c}.
\end{cases}
\end{equation}

\subsection{Escape times}
\label{ss:escape times}

Let  \( J \) be an arbitrary  interval with \( |J|<
\delta \).
We construct a countable partition
\( \mathcal P=\mathcal P(J) \) of \( J \) into
subintervals, called the \emph{escape partition} of \( J \), and
a stopping time function \( E: J\to \mathbb N \),
constant on elements of \( \mathcal P \).
Each element \( \omega \in \mathcal P \) has some combinatorial
information attached to its orbit up to time \( E(\omega) \) and
satisfies
\[
|f^{E(\omega)}(\omega)|\geq\delta.
 \]
We define the construction inductively as follows. Fix \( n\geq 1 \)
and suppose that a certain set of subintervals of \( J \)
have been defined for which \( E < n \). Let \( \omega \) be a
component of the complement of the set \( \{x\in J: E(x) < n\} \).

\subsubsection*{Inductive assumptions}
We suppose inductively that the following combinatorial information
is also available, the meaning of which will become clear when the
general inductive step of the construction is explained below:
\begin{itemize}
    \item every iterate \( i=1,\ldots, n \) is classified as either a
    \emph{free} iterate or a \emph{bound} iterate for~\( \omega \).
    \item the last free iterate before a bound iterate is called
    either an \emph{essential return} or an \emph{inessential return}.
    \item associated to each essential and inessential return there is a
    positive integer called the \emph{return depth}.
 \end{itemize}
We now consider various cases depending on
the length and position of the interval \({
\omega_{n}=f^{n}(\omega)} \) and on whether \( n \) is a free or bound
iterate for \( \omega \).

\subsubsection*{Escape times}
If  \( n \) is a free time for \( \omega \) and $|\omega_n|\geq\delta$
we say that $\omega$ has \emph{escaped}.  We  let
$\omega\in\mathcal P$ and define \( E(\omega) = n
\).
We call $\omega_n$ an escape interval.

\subsubsection*{Free times}
If \( n \) is a free time for \( \omega \) and $|\omega_{n}|<\delta$
we distinguish three cases:
\begin{enumerate}
\item If $\omega_n\cap\Delta=\emptyset,$
we basically do nothing: we
do not subdivide \( \omega \) further, do not add any combinatorial
information, and define \( n+1 \) to be again a free iterate for \(
\omega \).
\item If $\omega_n\cap\Delta\neq\emptyset$ but $\omega_n$ does not intersect
more than two adjacent $I_{r,j}$'s,
we do not subdivide \( \omega \) further at this moment, but add some
combinatorial information in the sense that we
say that $n$ is an \emph{inessential}
return time with \emph{return depth} $r$ equal to the minimum \( r \)
of the intervals \( I_{r,j} \) which \( \omega_{n} \) intersects.
Moreover we  define all iterates \( j=n+1,\ldots, n+p \) as
\emph{bound iterates} for \( \omega \) (\( \omega \) does not get
subdivided during these iterates, see below), where \( p=p(r) \) is
the binding period associated to the return depth \( r \) as defined
in \eqref{ss:binding}.
\item If $\omega_n\cap\Delta\neq\emptyset$ and
 $\omega_n$ intersects more than three adjacent
 $I_{r,j}$'s we subdivide
 \( \omega \) into subintervals $\omega_{r,j}$ in such a way that
each $\omega_{r,j}$ satisfies
$$I_{r,j}\subset f^n\omega_{r,j}\subset\hat{I}_{r,j}.$$
We say that $\omega_{r,j}$ has an \emph{essential} return at time
$n$, with return depth $r$ and define the corresponding binding period
as in the previous case.
\end{enumerate}

\subsubsection*{Bound times}
If \( n \) is a bound time for \( \omega \) we also basically do
nothing.  According to the construction above, \( n \) belongs to some
binding period \( [\nu+1, \nu+p] \) associated to a previous essential
or inessential return at time \( \nu \). So, if \( n< \nu+p \) we say
that \( n+1 \) is (still) a bound iterate, if \( n=\nu+p \) then \( n+1 \)
is a free iterate.

\subsection{Returns following escape times}
The notion of an escape time is meant
to formalize the idea that the interval in question has reached
large scale, and one intuitive consequence of this is that it should
therefore ``soon'' make a return to the domain of the inducing scheme.  

\begin{lemma}[{\cite[Lemma 1]{DiaHolLuz06}}]
\label{retpart}
 There exists $\delta^{*}, t^{*}, \xi>0$, all depending on~\( \delta \),
such that for
\( \Delta^{*}_f=(c^{*}_f-\delta^{*}, c^{*}_f+\delta^{*}) \) and for any
interval \( \tilde\omega \subset I\) with \( |\tilde\omega|\geq \delta \),
there exists a subinterval \( \tilde\omega^*\subset\tilde\omega \) such that:
\begin{enumerate}
\item $f^{t_0}$ maps $\tilde\omega^*$ diffeomorphically onto
$\Delta^{*}_f$ for some $t_{0}\leq t^{*}$,
\item $|\tilde\omega^*|\geq \xi |\tilde\omega|$,
\item both components of $\tilde\omega\setminus\tilde\omega^*$ are of size
$\geq\delta/3.$
\end{enumerate}
Moreover, the constants, and in particular the constant \( \xi \), can be chosen uniformly in a neighbourhood of \( f\) in \( \mathcal F\). 
\end{lemma}

The uniformity of the constants is not mentioned explicitly in  \cite[Lemma 1]{DiaHolLuz06} but follows immediately from the proof which we reproduce here for completeness and to highlight this property.

\begin{proof}
By assumption the preimages of $c^*$ are dense in
$I$ and do not contain any other critical point. Therefore for
any \( \varepsilon>0 \) there exists a \( t^{*} \) such
that the set of preimages
\(
\{f^{-t}(c^*): t\leq t^{*}\}
\) of the critical point \( c^{*} \) is
\emph{i)} \( \varepsilon \) dense in \( I \), and
\emph{ii)}  uniformly bounded away from
$\mathcal{C}$.
Using the \( \varepsilon  \)-density and taking \( \varepsilon \)
small enough (depending on \( \delta \) but not on \( \tilde\omega \))
we can guarantee that one of these preimages belongs to \(
\tilde\omega \) and in fact we can ensure that it lies arbitrarily close to
the center of~\( \tilde\omega \). Then, using that fact that these preimages
are uniformly bounded away from \( \mathcal C \) and taking \(
\delta^{*} \) sufficiently small we can  guarantee that a
component of
\( f^{-t_{0}}(\Delta^{*}) \) for some  \(0\leq t_{0}\leq t^{*}\)
is contained in the central third of \(\tilde \omega \). Since everything
depends only on a fixed and finite number of intervals and iterations
it follows that the proportion \( \xi \) of this preimage in \( \tilde\omega \) is
uniformly bounded below. By continuity we can choose $\delta^{*}, t^{*}, \xi>0$ constant in a neighbourhood of \( f\) in \( \mathcal F \).  
\end{proof}

\begin{remark}
The statement of Lemma \ref{retpart}  is slightly stronger than that of \cite[Lemma 1]{DiaHolLuz06} because we take advantage of our stronger assumption in (H3) that there exists a critical point whose preimages are dense in the entire interval \( I \). This allows us  to obtain the conclusions in the lemma for any interval \( \tilde\omega \subset I \) and is indeed the only place  where this assumption is required. 
\end{remark}

We have given the complete algorithm for the construction of the
escape partition \( \mathcal P \) of an arbitrary interval  \( J \). 
It is shown in \cite{DiaHolLuz06} that  this
algorithm not only gives rise to a partition \( \mathcal P \) of \( J \)
(mod 0) but in fact escapes occur exponentially fast.

\subsection{The induced Markov map}
\label{inducedMark}
We are now ready to describe the algorithm for the construction of the
final Markov induced map.
We fix \( \Delta^{*} \) as in Lemma \ref{retpart} and aim to obtain a
map  \( F: \Delta^{*}\to\Delta^{*}\) with
a partition $\mathcal{Q}$ and a return time function
$T:\mathcal{Q}\to\mathbb{N}$ constant on elements of $\mathcal{Q}$
such that \( F(\omega)=f^{T(\omega)}(\omega)=\Delta^{*} \) for every \(
\omega \in \mathcal Q \).

First of all, starting with \( \Delta^{*} \),
we construct the escape time partition \( \mathcal P(\Delta^{*}) \) as
described in Section \ref{ss:escape times}.
Let \( \omega \in \mathcal P(\Delta^{*}) \) with some escape time \(
E(\omega) = n \).
By Lemma~\ref{retpart}, we can subdivide its image
$\omega_{n}=f^{n}(\omega)$ into three pieces
\[
\omega_{n}= \omega_{n}^{L}\cup \omega_{n}^{*}\cup \omega_{n}^{R}
\]
with
\[
\omega_{n+t_{0}}^{*}=f^{n+t_{0}}(\omega)=f^{t_0}(\omega_{n}^{*})=\Delta^{*}
\]
for some $t_{0} \leq t^{*}$, and
\[
|\omega_{n}^{L}|,|\omega_{n}^{R}|>\delta/3.
\]
The interval \( \omega^{*} \) becomes, by definition,  an
element of $\mathcal{Q}$ and we define
$$
T(\omega^{*})=E(\omega)+t_0(\o)=n+t_0(\o).
$$
The components
$\omega_{n}^{L}, \omega_{n}^{R}$ are treated as new starting
intervals and we repeat the algorithm: we construct an escape
partition of each of $\omega_{n}^{L}, \omega_{n}^{R}$  and then some
proportion of each escaping component returns to \( \Delta^{*} \)
within some uniformly bounded number of iterates.  Notice that if
either \(|\omega_{n}^{L}|\geq \delta\)  or \( |\omega_{n}^{R}|\geq
\delta \) we can skip the construction of the escape partition (or, in
some sense, this step is trivial) and immediately apply Lemma \ref{retpart}
to find a subinterval
which returns to \( \Delta^{*} \) after some finite number of iterates
bounded by \( t^{*} \). As far as the construction is concerned we
only apply the escape partition algorithm to
intervals \( J \) of length between \( \delta/3  \) and \( \delta \).

\section{Uniqueness}\label{sec:unique}
We can now prove  Theorem \ref{thm:unique}. 

\begin{lemma}\label{lem:return}
For \(m\) almost all   \( x\in I \), there exists  a neighbourhood \( V\) of \( x \) and \( N > 0 \) such that  \(f^{N}(V)\subseteq  \Delta^{*}\). 
\end{lemma}
\begin{proof}
Notice that the  construction of the escape time partition, and indeed, the construction of the full induced map can actually be carried out on the entire interval \( I \), by trivial modifications of the arguments given above. Indeed, starting with a partition of \( I \) given by the critical partitions in the critical neighbourhoods, and subdividing the complement of the critical neighbourhoods into subintervals of length between \( \delta/3\) and \( \delta\), we can repeat exactly all the steps described above. Thus we can obtain a full branch Gibbs-Markov induced map \( \widehat F: I \to \Delta^{*}\).  Iterating \( \widehat F \) we get arbitrarily fine partitions of \( I \) whose elements map bijectively to \( \Delta^{*}\), which implies in particular the statement of the Lemma. 
\end{proof}

\begin{proof}[Proof of Theorem \ref{thm:unique}]
By the construction in \cite{DiaHolLuz06}  there exists a full branch induced Gibbs-Markov with integrable return times with inducing domain \( \Delta^{*}\). It follows that the measure   \( \mu_{f}\) is \emph{equivalent}  to Lebesgue in \( \Delta^{*}\). It is therefore sufficient to show that any  invariant absolutely continuous probability measure \( \nu \) also gives positive measure to \( \Delta^{*}\) which then implies that \( \nu = \mu_{f}\), see \cite[Lemma 3.12]{Alv20}. To see that this is the case, by  Lemma \ref{lem:return} there is a point \( x \)  in the support of \( \nu \), a neighbourhood \( V \) of \( x \), and \( N>0 \), such that   \(f^{N}(V)\subseteq  \Delta^{*}\). Since \( x \) is in the support of \( \nu \) we have \( \nu (V) >0\), and since \( \nu \) is invariant this implies that \( \nu (f^{N}(V))>0\) and therefore \(\nu ( \Delta^{*})>0 \).  
\end{proof}

\section{Small Return Times}\label{sec:small}
We can now complete the proof of Proposition \ref{prop:main2}. 

\begin{proof}[Proof of Proposition \ref{prop:main2}] 
 Notice first of all that the closeness of two maps \(f, g \in \mathcal F \) does not necessarily imply that the critical points  \( c^{*}_{f}, c^{*}_{g}\) given by condition (H3) are the corresponding points in the ordering~\eqref{eq:critorder}, and therefore does not imply that \( c^{*}_{f}, c^{*}_{g}\) are close. However, the conclusions of   Lemma \ref{retpart} are \emph{open} with respect to the metric on the family \( \mathcal F \) and therefore continue to hold if we replace the interval \( \Delta^*_f\) by the corresponding neighborhood \( \Delta_g\) of the corresponding critical point \(c_g\) which is close to \( c_f\) in the ordering~\eqref{eq:critorder}  even though this may, a priori, not be the point with dense preimages for \( g\) given by condition (H3). In particular we can carry out the  construction of the induced map for \( g \), as described in Section \ref{escape_part}, on the  \( \Delta_g\) instead of \( \Delta^*_g\). This allows us to  assume, in order to simplify the notation,  that \(c^*_g\) is in fact the corresponding critical point to \( c^*_f\) and therefore close to \( c^*_f\) for \( g \) close to \( f\) and that therefore the induced map for \( g \) is constructed on the domain \( \Delta^*_g\) close to \( \Delta^*_f\).

 We thus have defined two induced maps \( F:\Delta^*_f\to \Delta^*_f\) and  \( G:\Delta^*_g\to \Delta^*_g\) with inducing time functions \( T_f, T_g\) respectively. We fix \( N \geq 1 \) and \( \epsilon > 0 \) as in Proposition \ref{prop:main2} and consider the corresponding level sets \(\{T_{f}=j\} \) and \(  \{T_{g}=j\}\) for \( j \leq N \). Recall that the construction of the induced map is based on the notions of binding period, free period, return time and escape time, all of which depend continuously on the dynamics for a uniformly bounded number of iterates.  Therefore, since  the maps \( f, g \) can be made arbitrarily close, it follows that the level sets \(\{T_{f}=j\} \) and \(  \{T_{g}=j\}\), which are unions of intervals defined through these notions,  can also be made arbitrarily close in terms of the Lebesgue measure of their symmetric difference. 
 \end{proof}

%The key observation is that for all \( 0 < j \leq N \) we have  
%\[
%\{ T_{f}=j\}  \subseteq \{p_{f}< N \}
%\]
%

\section{Uniform Expansion}
\label{sec_bind}

We now begin the proof of Proposition \ref{prop:main} by showing that the constant \( \sigma\) related to the expansivity of the induced map, recall \eqref{ref:unif1}, can be chosen uniformly in the family \( \mathcal F \). 
From the construction of the induced map it follows  that the orbit of every \( x\in \Delta^* \) up to its inducing time is formed by pieces of orbit which are either outside \( \Delta^{c}_{f}\) or in a binding period. Outside \( \Delta^{c}_{f}\) we have the expansivity given by assumption  (H1)  and, following returns to \( \Delta^{c}_{f}\), we have a binding period which yields expansivity estimates as follows. 

\begin{lemma}[{\cite[Lemma 2]{DiaHolLuz06}}]\label{bindexp}
    There exist constants $\theta, \hat\theta>0$ 
%    independent of \( \delta \)
    such that for all points \( x\in \hat
    I_{r}, \) and $p=p(r) \geq 0$ we have
 \[
 |(f^{p+1})'(x)|\geq \frac{1}{\kappa} e^{\theta r}\geq
 \frac{1}{\kappa}e^{\hat\theta (p+1)}
\]
where \( \kappa>0 \) is the constant in the expansivity condition
 (H1). Moreover, the constants \( \theta, \hat\theta\) can be chosen uniformly in $\mathcal F$. 
\end{lemma}

\begin{proof}
The existence of \( \theta, \hat\theta\) is given in \cite[Lemma 2]{DiaHolLuz06} and it therefore just remains to show that \( \theta, \hat\theta\) can be chosen uniformly in \( \mathcal F\). To see this, we observe that it follows from the proof of  \cite[Lemma 2]{DiaHolLuz06}, that  \(
 \theta=\min\{\theta_{c}: c\in\mathcal{C}^{c}\}
 \) where \( \mathcal{C}^{c}\) is the set of (one-sided) critical points,  
\( \theta_{c}= 1- {5\alpha\ell_{c}}/{\Lambda} \), where \( \ell_c\) is the order of the critical point given by \eqref{defnCmap1}, and \( \alpha, \Lambda\) are given in condition (H2).  Both \( \alpha \) and \( \Lambda\) are uniform on \( \mathcal F \) by assumption and so condition \eqref{eq:alpha}  guarantees that the constants \( \theta_c\) are uniformly bounded away from 0.  Hence  \( \theta\) can be chosen uniformly on \( \mathcal F \). Recalling from the \cite[Equation 6]{DiaHolLuz06} that \( p  \leq {2\hat\ell r}/{\Lambda} \), we conclude  that \(\hat  \theta \) can be chosen (sufficiently small  between \( 0 \) and \( \theta\Lambda/2\hat\ell \)) uniformly in~\( \mathcal F \). 
\end{proof}

Combining the estimates in (H1)  and Lemma \ref{bindexp} we get that, for any \( x\in \Delta^* \)  and \( T=T(x)\),
\begin{equation}\label{property1}
|(f^{T})'(x)| \geq  \min\{e^{\lambda T}, e^{\hat\theta T}\}\geq \min\{e^{\lambda}, e^{\hat\theta}\}=:\sigma> 1. 
\end{equation}
Since $\lambda$ is given by (H1) and $\hat\theta$ can be chosen uniformly, it follows that  $\sigma$ can be chosen uniformly in \( \mathcal F\). 

\section{Uniform Distortion}
\label{distest}

We continue the proof of Proposition \ref{prop:main} by showing the uniformity of the distortion bound \( \mathcal D \) in~\eqref{ref:unif1}. For a given map \( f\in \mathcal F \) this  is  obtained  in \cite[Proposition 3]{DiaHolLuz06}. The proof is quite technical and requires the introduction of a number of intermediate constants \( \mathcal D_1\) to~\(\mathcal D_{11}\), which make up the final constant \( \mathcal D \). We will refer to the various sections of the proof of \cite[Proposition 3]{DiaHolLuz06}  and argue that each such constant can be chosen uniformly in~\( \mathcal F \), thus implying the same for \( \mathcal D \). 

The first constant \( \mathcal D_1\) comes from an argument in \cite[Section 5.2]{LuzTuc99} asserting that for 
all $x, y \in f(\hat I_{r})$, recall the construction in Section \ref{critpart}, and $1\leq k \leq p$, i.e. during the binding period, 
\begin{equation*}\label{binddist}
\left| \frac{(f^{k})'(x)}{(f^{k})'(y)} \right| \leq \mathcal
D_{1}.
\end{equation*}
 It can be seen from the proof of \cite[Proposition 5.4]{LuzTuc99} that \( \mathcal
D_{1} \)  depends only on \( \alpha\) and therefore can be chosen uniformly in \( \mathcal F \). 

To introduce the other constants, we now present the general framework. 
 Let \( \omega\subset \hat
J \) be an arbitrary interval, \( n\geq 1 \) an integer such
that \( \omega \) has a sequence \( t_{0},\ldots, t_{q} \leq n \) of
\emph{free} returns to \( \Delta \) (with respective return depth
sequence \( r_{t_0},\ldots,r_{t_q} \))
followed by corresponding
binding periods \( [t_{m}+1, t_{m}+p_{m}] \),  as described above.
In particular, for \( m=1,\ldots, q \), the interval \(
\omega_{t_{m}} \) is contained in the union of three adjacent
elements of the form \( I_{r,j} \) of the critical partition \(
\mathcal I \) of \( \Delta \). 
For $j\geq 0$, let $x_j=f^{j}(x),\ y_j=f^{j}(y)$,
$\omega_j=f^j(\omega)$ and fix  \( 1\leq k \leq n \). 

The second constant \( \mathcal D_2\) appears in \cite[Lemma 3]{DiaHolLuz06}
which says that 
\begin{equation}\label{bd3}
    \log\left|
    \frac{(f^{k})'(x_0)}{(f^{k})'(y_0)}\right| \leq \mathcal D_{2}
    \sum_{j=0}^{k-1} \frac{|\omega_{j}|}{\mathfrak D (\omega_{j})},
\end{equation}
    where \( \mathfrak D (\omega_{j}) := \sup\{\mathfrak D(x): x\in \omega\} \).  
This comes from the observation that
by the chain rule and the convexity of the
$\log$ function we have
\begin{equation}\label{bd1}
\log\left| \frac{(f^{k})'(x_0)}{(f^{k})'(y_0)}\right| =
\sum_{j=0}^{k-1}\log\left| 1+\frac{f'(x_j)-f'(y_j)}{f'(y_j)}
\right|\leq \sum_{j=0}^{k-1} \frac{|f'(x_j)-f'(y_j)|}{|f'(y_j)|}
\leq 
\sum_{j=0}^{k-1} 
\frac{|f''(\xi_{j})|}{|f'(y_j)|} |\omega_{j}|.
\end{equation}
where   \( \xi_j \in \omega_j\) is given by the Mean Value Theorem. 
Then, from  \eqref{defnCmap1}  we have 
\[ 
\frac{|f''(\xi_{j})|}{|f'(y_j)|} 
\leq C^2
\frac{|\xi_j-c|^{\ell_c-2}}{|y_j-c|^{\ell_c-1}} = 
C^2
\left|\frac{\xi_j-c}{y_j-c}\right|^{\ell_c-1} \cdot \frac{1}{|\xi_j-c|}
\]
where \( C \) is a uniform constant. Recall that both \( \xi_j\) and \( y_j\) belong to the interval \( \omega_j\).
If \(\omega_j\) is outside the critical/singular neighbourhoods then the distance of both points to the critical point~\( c \) is uniformly comparable and we get a uniform bound.  
 If \( \omega_j\) is in a neighbourhood of the critical point \( c \) then by construction it is contained in an interval \( \hat I_r,j\) defined in Section \ref{critpart}, and in particular \( |\omega_j|\lesssim e^{-r}/r^2\). Therefore we can write 
\[ 
\left|\frac{\xi_j-c}{y_j-c}\right|^{\ell_c-1} \leq 
% \left(\frac{ |\xi_j-y_j| + |y_j-c|}{|y_j-c|}\right)^{\ell_c-1}
 \left(1+ \frac{ |\xi_j-y_j|}{|y_j-c|}\right)^{\ell_c-1}
\leq 
 \left(1+ C' \frac{e^{-r}/r^2}{e^{-r}}\right)^{\ell_c-1}
 =  \left(1+ C' \frac{1}{r^2}\right)^{\ell_c-1}
\] 
for some uniform constant \( C'>0\) coming from the definition of the intervals \( \hat I_{r,j}\) and, letting \( z_j\) be the endpoint of \( \omega_j\) such that \( |z_j-c| = \mathfrak D(\omega_j)\) we have  
\[
\frac{1}{|\xi_j-c|} 
= \frac{1}{|z_j-c|} \cdot \frac{ |z_j-c|}{|\xi_j-c|}
\leq  \frac{C''}{\mathfrak D(\omega_j)}
\]
since the ratio \( { |z_j-c|}/{|\xi_j-c|} \) can be bounded uniformly just like in the previous equation. Substituting the these bounds into the expressions above and then into \eqref{bd1} we get \eqref{bd3} for some uniform constant \( \mathcal D_2\) as required. 

%
%\begin{prop}\label{escapedist} There exist constants \( \mathcal
%D_{\delta} \) and \( \tilde{\mathcal D}_{\delta} \) depending on \(
%\delta \), and $\mathcal D$ and \( \tilde{\mathcal D} \) independent
%of  \( \delta \), such that for every interval \( \omega \) and
%integer \( n \) as described in the previous paragraph, for every \(
%k\leq n \) and  \( x,y\in\omega \) we have
%\begin{equation*}
%\biggl|\frac{(f^{k})'(x)}{(f^{k})'(y)}\biggr|\leq \mathcal
%D_{\delta} \quad \text{ and }\quad
%\biggl|\frac{(f^{k})'(x)}{(f^{k})'(y)}-1\biggr|\leq
% \frac{\tilde{\mathcal D}_{\delta}}{|\omega_{k}|}|f^{k}(x)-f^{k}(y)|.
%\end{equation*}
%Moreover, the constants \( \mathcal D_{\delta} \) and \(
%\tilde{\mathcal D}_{\delta} \) can be replaced by $\mathcal D$ and
%\( \tilde{\mathcal D} \) under the following constraints:
%\begin{enumerate}
%    \item
%    either we allow every
%\( x,y\in \omega \) but restrict to values of  \( k\leq t_{q}+p_{q}
%\);
%\item
%or we allow all \( k\leq n \), in particular,  \( t_{q}+p_{q}\leq k
%\leq n\) but restrict to \( x,y\in\tilde\omega \)
%     where \( \tilde \omega \subset \omega \) is such that
% \( \tilde\omega_{j}\cap\Delta = \emptyset \) for
% \( k>j\geq t_{m}+p_{m} \),
%and  \( \tilde\omega_{k}\subset\Delta \); in this case we replace \(
%\omega_{k} \) by \( \tilde\omega_{k} \) in the second inequality.
%
%\end{enumerate}
%\end{prop}
%

The third constant \( \mathcal D_3\) appears in \cite[Lemma 4]{DiaHolLuz06}, where it is obtained by bounding the a sum 
\begin{equation*}
 \sum_{j=t_{m-1}+p_{m-1}+1}^{t_{m}-1} \kappa^{-1} e^{ -\lambda (t_{m}-j)}
\leq \mathcal D_{3}  |\omega_{t_{m}}| e^{r_{t_{m}}},
\end{equation*}
where the \( t\)'s are \emph{return depths} and the \( p \)'s are the corresponding binding periods. In particular, \( \mathcal D_3\) depends only on \( \kappa, \lambda\) both of which are fixed by condition (H1). 

The fourth constant \( \mathcal D_4\) appears in \cite[Lemma 5]{DiaHolLuz06} where  it is formulated in terms  of further constants \( \mathcal D_{4}=\mathcal
D_{7}\sum_{i=0}^{\infty}  e^{-\alpha i}\), wehere 
\( \mathcal D_{7}=\mathcal
D_{5}\mathcal D_{6} \), which are introduced in the proof of the Lemma, and  \( \alpha\) is given by (H2).  It is therefore sufficient to show that the constants \(\mathcal
D_{5}, \mathcal D_{6} \) can be chosen uniformly. 
The constant \( \mathcal D_5\) appears in \cite[Sublemma 5.1]{DiaHolLuz06} and it is clear from the proof of this sublemma that \( \mathcal D_5\)  depends only on \( \mathcal D_1\) and the constant   in
 \eqref{defnCmap1}. 
The constant \( \mathcal D_6\) appears in \cite[Sublemma 5.2]{DiaHolLuz06} and is defined explicitly as \( \mathcal D_6 = 1-e^{-\alpha}\). 

Moving on, the constant \( \mathcal D_8\) is defined in  \cite[Subsection 4.4]{DiaHolLuz06} as the sum of \( \mathcal D_3\) and~\( \mathcal D_4\) and \( \mathcal D_9\) is obtained in the proof of  \cite[Sublemma 5.3]{DiaHolLuz06} as a bound for the geometric sum  \( \sum_{i=0}^{\infty}  e^{-\theta r_\delta i}\) where \( \theta \) is the uniform constant given in Lemma \ref{bindexp} above, and \( r_\delta = \log\delta^{-1}\) which is uniform by definition, recall \eqref{eq:alpha}. The constant \( \mathcal D_{10}\) is defined explicitly at the end of \cite[Subsection 4.4]{DiaHolLuz06} as \( \mathcal D_{10}= \mathcal D_3+ \mathcal D_8\mathcal D_9\sum_r 1/r^2\). Finally, the constant \( \mathcal D_{11}\) is obtained in \cite[Subsection 4.5]{DiaHolLuz06} depending only on the two  constants \( \mathcal D_2\) and \( \mathcal D_{10}\).

We have therefore verified that all the constants \( \mathcal D_1\)-\( \mathcal D_{11}\) can be chosen uniformly in the family \( \mathcal F \) and thus the same is true for  the distortion constant \( \mathcal D \) of Proposition \ref{prop:main}.

\section{Tail of Inducing Time}\label{ret_time}
It only remain to prove that the constants \( C, \gamma\), which bound the tail of the inducing time in~\eqref{eq:Cgammatail} can be chosen uniformly in \( \mathcal F \). Recall first of all, from Section \ref{inducedMark}, that the induced map \( F: \Delta^{*}\to\Delta^{*}\)  has a partition \( \mathcal Q \) and a return time function
$T:\mathcal{Q}\to\mathbb{N}$ constant on elements of $\mathcal{Q}$
such that \( F(\omega)=f^{T(\omega)}(\omega)=\Delta^{*} \) for every \(
\omega \in \mathcal Q \).  We let 
\[
\mathcal{Q}^{(n)}=\{\omega\in \mathcal Q: T(\omega)>n\}
 \]
In  \cite[Proposition 2]{DiaHolLuz06}  it is shown that there are constants \( C_2, \gamma_2\)  such that 
\begin{equation*}
|\mathcal{Q}^{(n)}|\leq C_2e^{-\gamma_2 n}|\Delta^{*}|.
\end{equation*}
Our bound in \eqref{eq:Cgammatail} then follows by choosing \( C= C_2\) and \( \gamma=\gamma_2\). 
It is therefore sufficient to show that \( C_2, \gamma_2\) can be chosen uniformly in \( \mathcal F \) 
We will refer to the argument in \cite[Section~5 and~6]{DiaHolLuz06} to show that this is the case.

By construction, each \( \omega\in \mathcal Q^{(n)} \)
is contained in a nested sequence of intervals
\[
\omega\subset\omega^{(s)}\subset\omega^{(s-1)}
\subset\ldots\subset\omega^{(1)} \subset \Delta^{*}
\]
corresponding to escape times $E_1,\ldots,E_s$, such that \(
|f^{E_i}(\omega^{(i)})|\geq\delta \) for  \( i=1\ldots, s \). This
sequence is empty for those elements of \( \mathcal Q^{(n)} \) which
have not had any escape before time~\( n \) (such as those which
start very close to the critical point). For \( s = 0,\ldots, n \)
we let \(
\mathcal{Q}^{(n)}_s \) denote the collection of intervals in \(
\mathcal{Q}^{(n)} \) which have exactly \( s \) escapes before time
\( n \). Then for   any constant $\zeta\in (0,1)$ we write
\begin{equation}\label{set_split}
|\Q^{(n)}| =\sum_{s\leq n}|\Q_s^{(n)}|= \sum_{s\leq \zeta
n}|\Q_s^{(n)}| + \sum_{\zeta n<s \leq  n}|\Q_s^{(n)}|.
\end{equation}
This corresponds to distinguishing those intervals which
have had lots of escape times and those
that have had only a few.  It is proved in \cite[Lemma 10 and 11]{DiaHolLuz06} that there exist constants \( C_3, C_4, \gamma_3, \gamma_4 > 0\) such that for sufficiently small \( \zeta>0\) and for all \( n\geq 1 \), 
\begin{equation}\label{c3c4}
     \sum_{0\leq s\leq \zeta n}|\Q_s^{(n)}|  \leq C_{3}
e^{-\gamma_{3} n} |\Delta^{*}|
\quand 
  \sum_{\zeta n<s \leq  n}|\Q_s^{(n)}|  \leq C_{4}
e^{-\gamma_{4} n} |\Delta^{*}|.
\end{equation}
It is therefore sufficient to show that the constants \( C_3, C_4, \gamma_3, \gamma_4 \) can be chosen locally uniform. The first step is explain how these constants appear in \cite{DiaHolLuz06}. 
The proof of \cite[Lemma 11]{DiaHolLuz06}  shows that  
\begin{equation}\label{many_esc}
\sum_{\zeta n<s<n}|\Q_s^{(n)}|\leq \sum_{\zeta
n<s<n}\left(1-\frac{\xi}{\mathcal D}\right)^{s-1}
|\Delta^{*}|
\leq \frac{\mathcal D}{\xi}\left(1-\frac{\xi}{\mathcal D}\right)^{\zeta n}|\Delta^{*}|.
\end{equation}
which clearly gives the second inequality in \eqref{c3c4} for the obvious choices of \( C_4, \gamma_4\). 
For  the first inequality in \eqref{c3c4}, the proof of \cite[Lemma 10]{DiaHolLuz06} shows that 
    \begin{equation}\label{eq:firstineq}
  \sum_{0\leq s\leq\zeta n}|\Q_s^{(n)}|
    \leq \sum_{0\leq s\leq\zeta n}N_{n,s}
  \frac{C_1}{\delta^{*}}  
  \left(\frac{C_1\mathcal D}{\delta}\right)^s
     e^{-\gamma_1 n}|\Delta^{*}|,
    \end{equation}
where $N_{n,s}$ is the number of possible integer sequences
$(t_1,\ldots, t_{s})$ such that $\sum t_j=n$. 
The number   of such
 sequences is the number of ways to choose $s$ balls from a row of $k+s$
 balls, thus partitioning the remaining $k$ balls into at most $s$
 disjoint subsets.  Therefore, using also that \( s\leq
  \zeta n \), we have 
 \[
 N_{n,s}\leq\begin{pmatrix}
 n+s\\ s
 \end{pmatrix}
 =\begin{pmatrix} n+s\\ n
 \end{pmatrix} \leq 
 \begin{pmatrix}
 (1+\zeta)n \\ n
 \end{pmatrix}
 =\frac{[(1+\zeta)n]!}{(\zeta n)! n!}.
\]
 Using Stirling's  formula
 $k!\in [1,1+\frac{1}{4k}]\sqrt{2\pi k}k^{k}e^{-k}$
 we obtain
\begin{equation}\label{ns}
 N_{n,s}
 \leq\frac{[(1+\zeta)n]^{(1+\zeta)n}}{(\zeta n)^{\zeta n}n^n}
% =(1+\zeta)^{(1+\zeta)n}\zeta^{-\zeta n}
 %& \leq \exp\{(1+\zeta)n\log(1+\zeta) -\zeta n\log\zeta\}
 \leq \exp\{((1+\zeta)\zeta -\zeta \log\zeta)n\}=e^{\hat\zeta n}
\end{equation}
 where 
 \[
 \hat\zeta =
 \bigl((1+\zeta)\zeta -\zeta \log\zeta\bigr).
 \]
Replacing this in \eqref{ns} and then in \eqref{eq:firstineq} gives \footnote{Note that the term \( {C_1 \mathcal D}/{(C_1\mathcal D - \delta)}\) is missing by mistake in the last displayed formula of \cite[Subsection 6.1]{DiaHolLuz06}, though this has no effect on the results there (and here). 
} 
\begin{equation}\label{c3}
\sum_{0\leq s\leq\zeta n}|\Q_s^{(n)}| \leq
\frac{C_{1}}{\delta^{*}} 
\frac{C_1 \mathcal D}{C_1\mathcal D - \delta}
\cdot
\left(\frac{C_{1}\mathcal
D}{\delta}\right)^{\zeta n} e^{\hat \zeta n}
e^{-\gamma_{1} n} |\Delta^{*}|.
\end{equation}
By choosing  \( \zeta \) sufficiently small,  and thus making  \( \hat\zeta \) arbitrarily small,  \eqref{c3}  gives the first inequality in \eqref{c3c4} for the obvious choices of \( C_3, \gamma_3\). 

We have thus shown that the constants  \( C_3, C_4, \gamma_3, \gamma_4\) depend on a number of other constants, as shown in the first two lines of Table~\ref{dep} below. The remainder of the table refers to appropriate places in this paper and in  \cite{DiaHolLuz06} to establish  a chain of dependencies which ultimately reduces to ``primary" constants which we know are uniform, either by assumption or by the arguments in the previous sections. 

%
%
%\medskip
%
%
%\begin{lemma}
%% Let $N_{R, s}$ denote the number of integer sequences
%% $(t_1,\ldots,t_s)$, with $s \leq \eta R$  and $\sum_{i=1}^{s}t_i = R$. 
%% Then 
% \(
% N_{n,s} \leq e^{\hat\zeta  n}
% \),
% where 
% $\hat\zeta =
% (1+\zeta)\zeta -\zeta \log\zeta$.
% \end{lemma}

 %\begin{proof} 
 
% \end{proof}

\medskip

\begin{center}
\begin{table}[h]
\begin{tabular}{|c|c|c|}
  \hline
  Constant & Dependence & Reference \\ \hline
  %after \\: \hline or \cline{col1-col2} \cline{col3-col4} ...

$ C_3, \gamma_3$ &  $C_1, \mathcal D, \delta^*, \delta, \zeta, \hat\zeta, \gamma_1 $ & Equation \eqref{c3} \\

 $ C_4, \gamma_4$ &  $ \xi, \mathcal D, \zeta $ & Equation \eqref{many_esc}\\
 
     $ \zeta, \hat\zeta $ &  $ C_1, \mathcal D,  \delta,  \gamma_1 $ & Equation \eqref{c3} \\
  
      $ $ &  $ $ & \\
      
 $ C_1, \gamma_1$ &  $\tilde C_1, \theta, \tilde \theta, \tilde\eta, n_\delta $ & \cite[Proposition 1, Equation after Lemma 9]{DiaHolLuz06} \\

         $ \tilde C_1$ &  $\theta, \tilde\theta, \tilde\eta, n_\delta  $ &  \cite[Proposition 1, Equation after Lemma 9]{DiaHolLuz06} \\
    
      $ \tilde\eta$ &  $\eta_1, \eta_2, \eta_3 $ & \cite[End of Section 5.2]{DiaHolLuz06} \\

     $ \tilde\theta$ &  $\eta_3, \hat\lambda $ & \cite[End of Section 5.4]{DiaHolLuz06}\\

    $ n_\delta$ &  $\delta, \lambda, \hat\lambda, \kappa $ & \cite[Proof of Lemma 8]{DiaHolLuz06} \\

        $ \eta_1$ &  $ \eta$ & \cite[Section 5.2, end of Page 445]{DiaHolLuz06} \\

    $ \eta_2$ &  $ r_\delta, N_c$ & \cite[Equation (24)]{DiaHolLuz06}  \\

        $ \eta_3$ &  $ r_\delta $ &  \cite[Equation (25)]{DiaHolLuz06}  \\

    $ \eta$ &  $r_\delta $ &  \cite[Section 5.2, Page 445]{DiaHolLuz06}  \\
    
   $ \hat\lambda, \tilde\lambda$ &  $ \lambda, \hat\theta, \Lambda$ & \cite[Section 5.3]{DiaHolLuz06} \\
     
              $r_\delta  $ &  $ \delta$ & Condition \eqref{rdelta}\\

    $ $ &  $ $ & \\
    
     $ \xi $ & locally uniform  &  Lemma \ref{retpart} \\

      $  \delta^*$ & uniform  &  Lemma \ref{retpart} \\
 
 $\mathcal D $ & uniform &  Section \ref{distest} \\
 
 $ \theta, \hat\theta$ &  uniform & Lemma \ref{bindexp}\\
 
       $ $ &  $ $ & \\

     $ N_c$ & uniform & Assumption \eqref{card} \\

 $\lambda, \Lambda, \kappa,  \delta$ &  uniform & Assumption \eqref{eq:alpha}\\
  
  \hline
\end{tabular}
\bigskip
\caption{\small Dependencies of constants}
\label{dep}
\end{table}
\end{center}

\newpage
%\bibliographystyle{amsplain}
%\bibliography{AlvGamLuz}


\end{document}